\documentclass[10pt,reqno,a4paper]{amsproc}
\usepackage{amsmath}
\usepackage{amssymb}
\usepackage[all]{xy}
\usepackage{psfrag}
\usepackage{graphicx}
\usepackage{enumerate}
\usepackage{algorithmic}
\usepackage{algorithm}

\swapnumbers
\theoremstyle{plain}
\newtheorem{theorem}{Theorem}[section]

\newtheorem{lemma}[theorem]{Lemma}
\newtheorem{cor}[theorem]{Corollary}

\theoremstyle{definition}

\newtheorem{defn}[theorem]{Definition}

\newtheorem{remark}[theorem]{Remark}

\newcommand{\abs}[1]{\left |{#1}\right |}

\newcommand{\set}[1]{\left\{{#1}\right \}}

\newcommand{\all}[2]{\left \{ {#1}\,:\,{#2} \right \}}

\newcommand{\C}{\mathbb C}
\newcommand{\R}{\mathbb R}
\newcommand{\Z}{\mathbb Z}
\newcommand{\N}{\mathbb N}
\newcommand{\T}{\mathbb T}

\usepackage[usenames, dvipsnames]{color}

\definecolor{mygray}{gray}{0.6}

\begin{document}

\date{24 July 2023}

\title{EDMD for expanding circle maps and their complex perturbations}

\author[O.F.~Bandtlow]{Oscar F.~Bandtlow}

\address{%
Oscar F.~Bandtlow\\
School of Mathematical Sciences\\
Queen Mary University of London\\
London\\
UK.
}
\email{o.bandtlow@qmul.ac.uk}

\author[W.~Just]{Wolfram Just}

\address{%
Wolfram Just\\
Institute of Mathematics\\
University of Rostock\\
Rostock\\
Germany.
}
\email{wolfram.just@uni-rostock.de}

\author[J.~Slipantschuk]{Julia Slipantschuk}

\address{%
Julia Slipantschuk\\
Mathematics Institute\\
University of Warwick\\
Coventry\\
UK.
}
\email{julia.slipantschuk@warwick.ac.uk}

\subjclass[2010]{Primary: 37M25; Secondary: 37C30, 37E10, 30H10}
\keywords{Koopman operator, expanding circle maps, EDMD, Galerkin method, collocation method}

\begin{abstract}
We show that spectral data of the Koopman operator arising from an analytic expanding
circle map $\tau$ can be effectively calculated using an EDMD-type algorithm combining a
collocation
method of order $m$ with a Galerkin method of order $n$. The main result is that if
$m\geq \delta n$, where $\delta$ is an explicitly given positive number quantifying by how much $\tau$ expands concentric annuli containing the unit circle, then
the method converges and approximates the spectrum of the Koopman
operator, taken to be acting on a space of analytic hyperfunctions,
exponentially fast in $n$.
Additionally, these results extend to more general expansive maps on suitable annuli
containing the unit circle.

\end{abstract}

\maketitle

\section{Introduction}

Identifying dynamically relevant signatures and effective degrees of freedom
is among the most challenging and fruitful tasks in science in general.
A specific method, termed dynamic mode decomposition (DMD), 
aims at combining the success of linear data analysis with dynamical systems theory. 
It exploits dynamical signatures which can be traced back
to eigenmodes of an evolution operator defined on a suitable
function space, thereby presenting a key
example for data analysis inspired by abstract operator theory.
While the term ``dynamic mode decomposition'' appears to have been coined in
\cite{ScSe_BAPS08}, it is hard to identify an exact single source for this approach.
Related ideas have been used to reconstruct invariant measures of
dynamical systems \cite{DeJu_JNA99}, and then considerably extended 
to estimate spectral properties of the Koopman operator from data
\cite{Mezi_ND05}. In \cite{WiKeRo_JNS15}, the DMD algorithm was extended 
to cope with a broader range of observations.
While there have been countless applications of DMD and its variants, here we restrict
ourselves to a brief overview of a few of the
major references, see \cite{BiMoMe_Chaos12,WiRoMeKe_EPL15,KKS,Gian_ACHA19,KNPNCS_20,BBKK, LDBK} and references therein.

The main idea of dynamic mode decomposition follows concepts 
developed in the context of statistical data analysis and statistical physics.
In particular, data-driven approaches to identify
empirical modes with linear techniques
can be traced back to \cite{Karh_AASF47}, which has served as a seed for a
wealth of linear decomposition techniques developed in the last
decades. 
Beyond a plain data analysis aspect, the identification of relevant dynamical signatures
adds to the theoretical understanding of motion.
From a mathematical perspective,
the role of phase space structures and, in particular, of the underlying
function spaces have been pointed out in the context of theoretical physics,
namely, as one of the main clues to understand
the emergence of irreversibility in an otherwise time-invariant Hamiltonian
setup \cite{PrMaGeHa_PNAS77}, see also \cite{SaHa_PLA92a} for an
illustration in an elementary setting.
While these works provide some mathematical tools to rigorously underpin
contemporary nonlinear data analysis techniques, identifying relevant dynamical
signatures in general far-from-equilibrium processes remains a substantial challenge,
and developing paradigms to identify effective degrees of freedom in general dynamical
systems continues to be a very active area of research.

We will briefly outline the formal aspects of extended
dynamic mode decomposition (EDMD) for a discrete dynamical system
$z_{t+1}=\tau(z_t)$, see \cite{WiKeRo_JNS15} for more explanations. 
Assuming that the dynamics is observed through
a collection of $N$ scalar functions $\{\psi_0,\ldots,\psi_{N-1}\}$ and
given a sequence of $m$ points in phase space $z^{(0)},\ldots,z^{(m-1)}$,
one constructs two $N\times N$ matrices
\begin{align*}
G_{k,l}&=\frac{1}{m}\sum_{j=0}^{m-1} \psi_k\left(\tau\left(z^{(j)}\right)\right)
\overline{\psi_l\left(z^{(j)}\right)}, \qquad (k,l=0, \dots, N-1), \\
H_{k,l}&=\frac{1}{m}\sum_{j=0}^{m-1} \psi_k\left(z^{(j)}\right)
\overline{\psi_l\left(z^{(j)}\right)}, \quad \quad \qquad (k,l=0, \dots, N-1), 
\end{align*}
and defines the matrix $M$ as
\begin{align}\label{eq:M}
M=G H^{-1}
\end{align}
(see also \eqref{eq:Mdef} in Section \ref{sec:2}). This matrix
provides an optimal finite-dimensional linear approximation to the original
nonlinear system\footnote{To be precise,
the matrix $M = G H^{-1}$ provides a least-squares solution $\arg \min_{M} \| MX - Y \|_2^2$ for $X = [\psi(z_0), \ldots, \psi(z_{m-1})]$
and $Y = [\psi(\tau(z_0)), \ldots, \psi(\tau(z_{m-1}))]$, with $\psi(z) = (\psi_0(z),\ldots, \psi_{N-1}(z))^T$.
This can be seen by noting that such solution is given by $M = Y X^+ = (Y X^H) (X X^H)^{-1} = G H^{-1}$ whenever $X X^H$ is invertible, with $X^+$ and $X^H$ denoting the Moore-Penrose pseudoinverse and conjugate transpose of $X$, respectively.},
and hence its eigendata are related to the time
scales of the original dynamical system.
It can be linked to the
Koopman operator $C_\tau$, the linear composition operator governing the underlying
dynamics given by
\begin{equation}\label{eq:C_tau}
C_\tau f= f \circ \tau
\end{equation}
on some suitable function space. In particular, the EDMD method aims to establish convergence of the
eigenvalues and eigenvectors of $M$ to those of the Koopman operator in the
limit of large number of observables, $N\rightarrow \infty$, and
large number of nodes in phase space, $m\rightarrow \infty$.

Theoretical considerations of EDMD
are often based on a Koopman operator defined on the space of square-integrable
functions, which however raises fundamental technical issues.
In the case of invertible measure-preserving dynamical systems, the Koopman operator defined on $L^2$ turns out to be
unitary, which provides the necessary mathematical machinery for its rigorous study in terms of spectral measures,
see for example \cite[Chapter~1]{walters} and \cite[Chapter~18]{EFHN}
or \cite{CT, KPM, ZZ} in the data-driven context.
With the notable exception of the Hamiltonian setting, however, the assumption of invertibility
appears to be very limiting, as in most applications a well-defined invertible global flow cannot
be established.
Even in cases where the Koopman operator is merely an isometry and can thus be decomposed into a unitary and a shift part by 
means of the Wold decomposition (see, for example, \cite[Chapter~1.3]{RR}),
a study focused solely on the unitary part may miss relevant
dynamical features of the underlying dynamical system.

In the case of dissipative dynamics, one can still define the Koopman operator (\ref{eq:C_tau}) on $L^2$,
and the general spectral theory and properties of the trivial eigenvalue still provide
essential information about ergodic and mixing properties of the system (see, for example, \cite[Chapter~1]{walters}).
However, in this setting the Koopman operator no longer enjoys nice spectral properties for
this choice of function space. 
For a non-invertible discrete dynamical system preserving a mixing probability measure,
the spectrum of the Koopman operator consists of the
entire unit disk (see, for example, \cite[Remark~4.4]{Ke2} or appeal to the Wold decomposition), 
and no nontrivial eigenvalues or eigenfunctions exist,
precluding spectral convergence results or, in fact, any theoretical analysis which assumes the existence
of eigenfunctions in $L^2$.
Nevertheless, for various chaotic dynamical systems with enough regularity, EDMD as given above empirically still yields spectral data which correctly determines the correlation decay rates for sufficiently regular observables.  
This can indeed be rigorously understood by analysing the residual spectrum of the Koopman operator
defined on $L^2$.
For a value $\lambda \in \mathbb{C}$ in the residual spectrum, the range of $\lambda -C_\tau$ is not the entire Hilbert space,
whereas its kernel is trivial as there are no eigenvalues in the residual spectrum. By enriching $L^2$ with ``generalised functions''
 one can make $\lambda -C_\tau$ surjective,
while for some selected values $\lambda \in \mathbb{C}$ the kernel of $\lambda -C_\tau$ becomes
nontrivial. In the particularly nice setting of hyperbolic analytic systems one observes that
the residual spectrum on this larger space vanishes in exchange for the
occurrence of a point spectrum, while the Koopman operator becomes a compact operator, see the discussion in 
\cite{SlBaJu_CNSNS20}.
The corresponding eigenvalues are precisely the decay rates observed in the dynamical
system, while the corresponding ``physical'' eigenfunctions are objects which are not
in $L^2$ (see \cite{F-CGT-Q} for a striking visual demonstration), and EDMD detects this subtle structure. Again we want to stress that any
analysis of the Koopman operator on $L^2$ which assumes the existence of nontrivial eigenfunctions
misses this essential point
and is problematic from a mathematical and physical perspective.

The results in non-equilibrium statistical mechanics alluded to earlier on 
suggest that the Koopman operator is best understood on physically relevant function spaces which guarantee 
compactness or, more generally, quasi-compactness of the operator and yield well-defined
relaxation rates \cite{SuAnTaBa_JMP96}.
Using these ideas, we will clarify in Section \ref{sec:2}
why EDMD is such a successful tool even in the case of chaotic dynamical systems,
restricting our attention to the simplest such systems, one-dimensional expanding
maps.
As this method produces discrete spectra, it is well-suited to approximate 
compact Koopman operators. In fact, as our main result (Theorem \ref{thm:main2}) will show,
EDMD implicitly treats compact Koopman operators on a suitably enriched
function space which contains highly singular objects.
Counterintuitively, it is precisely the presence of these highly singular objects that
is responsible for the physical decay rates detected
via EDMD.
The main novelty of our results is a simple constraint on the two key parameters in EDMD,
the number of observables and the number of nodes (that is, the size of the data set), which
is sufficient to guarantee exponential convergence to the exact spectrum of a compact Koopman operator
on the appropriate Hilbert space.

As the general statement of our main results requires some
background explained in Section \ref{sec:2}, we want to highlight its practical importance
by presenting the following simple corollary of one of our results (Theorem~\ref{thm:main3}), which applies to expanding maps on the circle.
Writing $\tau'_{\mathrm{min}} = \min_{z\in \T}|\tau'(z)|$
(and $\tau'_{\mathrm{max}} = \max_{z\in \T}|\tau'(z)|$), we say $\tau\colon \T \to \T$ is an expanding circle map
on $\T = \{z \in \C\colon |z| = 1\}$ if $\tau'_{\mathrm{min}} > 1$.

\begin{theorem}
\label{thm:intro}
Let $\tau$ be an analytic expanding circle map, $C_\tau$ the associated
Koopman operator given in \eqref{eq:C_tau}, and $\tau'_{\mathrm{min}}, \tau'_{\mathrm{max}}$
the minimal and maximal derivatives of $\tau$ on $\mathbb{T}$. Then the following holds. 
\begin{enumerate}
\item There exists a Hilbert space $\mathcal{H}$ such that $C_\tau$
is a well-defined compact operator from $\mathcal{H}$ to $\mathcal{H}$. In particular, its spectrum
$\operatorname{spec}(C_\tau)$ is either a finite set or a sequence
converging to zero together with zero itself and each non-zero spectral point is an eigenvalue of finite algebraic multiplicity.
\item For $m,n \in \mathbb{N}$
and $\alpha \in \mathbb{R}$ let $\{\psi_l\}_{-(n-1) \leq l \leq (n-1)}$ with
$\psi_l(z) = z^{l}$ be the set of
observables, $\{z^{(j)}\}_{0\leq j \leq m-1}$ with $z^{(j)}=e^{2\pi i j/ m + i\alpha}$ be
the set of phase space points 
and $M=M_n$ the $(2n-1) \times (2n-1)$ matrix given in \eqref{eq:M}.
If for every $n \in \mathbb{N}$, the number of phase space points $m = m(n) \in \mathbb{N}$ is chosen such that
\begin{equation}\label{eq:m2n_cond}
m \geq (\tau'_{\mathrm{min}} + \tau'_{\mathrm{max}}) n,
\end{equation}
then the following hold:
  \begin{enumerate}
  \item Any convergent sequence $(\lambda_n)_{n\in\mathbb{N}}$ with
  $\lambda_n \in \operatorname{spec}(M_n)$ converges to a spectral point of $C_\tau$.
  \item \label{it:b} Conversely, for any $\lambda \in \operatorname{spec}(C_\tau)$ there is a sequence
  $(\lambda_n)_{n\in\mathbb{N}}$ with $\lambda_n \in \operatorname{spec}(M_n)$ such that
  $\lambda_n \to \lambda$ as $n\to \infty$. Moreover,
  \[|\lambda - \lambda_n| = O(e^{-an}) \quad \text{as } n\to \infty,\] for some $a>0$. 
  \item Suppose $\lambda \in \operatorname{spec}(C_\tau)$ is non-zero and $(\lambda_n)_{n\in \N}$ denotes the approximating sequence of eigenvalues of $M_n$ given in (b). If $\xi_n=(\xi_{n,-n+1}, \ldots, \xi_{n,0},\ldots, \xi_{n, n-1})^T \in \mathbb{C}^{2n-1}$ is a generalised eigenvector of the transpose $M^T_n$ of $M_n$, then setting 
  \[ h_n(z)=C\sum_{\abs{k}<n}\xi_{n,k}z^k \quad (z\in \mathbb{C})\,, \]
  where $C$ is a constant chosen so that $\|h_n\|_{\mathcal{H}}=1$,  
  yields a sequence of Laurent polynomials $(h_n)_{n\in \mathbb{N}}$
  with  
  \[ \| \mathcal{P}h_n-h_n\|_{\mathcal{H}} = O(e^{-bn}) \quad \text{as } n\to \infty,\]
where $\mathcal{P}$ is the spectral projection associated to the eigenvalue $\lambda$ of $C_\tau$ and some $b>0$. 
  \end{enumerate}
\end{enumerate}
\end{theorem}

In short, for analytic expanding maps on the circle, this theorem guarantees exponential convergence
of eigenvalues and eigenvectors of the EDMD matrix in \eqref{eq:M} constructed using
$2n-1$ Laurent polynomials as observables and $m$ equidistant points in the phase space, to those
of the associated Koopman operator, if $m$ is a linear function of $n$ satisfying
\eqref{eq:m2n_cond}.
In the setting of non-invertible maps, it is more common (in the dynamical systems community) to approximate the adjoint of the Koopman operator, known as the Perron-Frobenius or transfer operator, as it can be considered on ordinary function spaces such as spaces of functions of bounded variation, Sobolev spaces, or spaces of analytic functions, see \cite{Ke1, DeJu_JNA99, Ke-Li, BaH} for a small snapshot of a large body of works or \cite{GMNP, BS, W} for some more recent contributions.
Whilst working with the transfer operator requires knowledge of the local inverse branches of $\tau$ and their derivatives,
EDMD only requires knowledge of the map itself, and can be leveraged as an alternative way
to approximate the spectrum of either of the operators.

In order to benchmark our results and the sharpness of our bound in \eqref{eq:m2n_cond} we will resort
in Section \ref{sec:3} to
a class of analytic maps where full spectral information is accessible
\cite{SlBaJu_NONL13,BaJuSl_AIHP17}, and which has been used recently
to clarify some convergence properties of EDMD \cite{SlBaJu_CNSNS20}.
Within this class of chaotic systems we demonstrate that
the rigorous convergence estimates are surprisingly sharp
(see Figure~\ref{fig:fig_nm} in 
Section~\ref{sec:3}), and that EDMD correctly identifies the
spectrum of the compact Koopman operator considered on a space of generalised functions. 

\section{Koopman operator for holomorphically expansive maps on spaces of generalised functions}\label{sec:2}
For $r\in (0,1)$ let $A_r=\all{z\in \C}{r<|z|<r^{-1}}$ denote an
annulus containing the unit circle $\mathbb{T}=\all{z\in \C}{|z|=1}$. Let $\tau$ be an \emph{analytic expanding circle map}, that is, $\tau:\mathbb{T}\to\mathbb{T}$ has an analytic extension to some annulus $A_r$ with $r\in(0,1)$ and $\min_{z\in \mathbb{T}}|\tau'(z)|>1$. We start by showing that the expansivity assumption on the unit circle entails a type of expansivity in the complex plane in the following sense. 
\begin{lemma}\label{lem:compexp}
Let $\tau$ be an analytic expanding circle map. Then there is $\tilde{r}\in (0,1)$, such that the following holds: 
\begin{enumerate}
\item $\tau$ is holomorphic on $A_{\tilde{r}}$. 
\item For every $r_1\in (\tilde{r},1)$ there are radii $r_2$ and $r_3$ with $0<r_3<r_2<r_1<1$ with the following property: 

if $\tau$ is
orientation-preserving then
\begin{equation}\label{eq:tauop1}
\begin{aligned}
r_3^{\hphantom{-1}}<&\abs{\tau(z)}<r_2^{\hphantom{-1}} \quad \text{for any $z\in\C$ with  $|z|=r_1$}, \\
r_2^{-1}<&\abs{\tau(z)}<r_3^{-1}   \quad \text{for any $z\in\C$ with  $|z|=r_1^{-1}$};
\end{aligned}
\end{equation}

while if $\tau$ is orientation-reversing then
\begin{equation}\label{eq:tauor1}
\begin{aligned}
r_2^{-1}<&\abs{\tau(z)}<r_3^{-1} \quad \text{for any $z\in\C$ with  $|z|=r_1$}, \\
r_3^{\hphantom{-1}}<&\abs{\tau(z)}<r_2^{\hphantom{-1}}   \quad \text{for any $z\in\C$ with  $|z|=r_1^{-1}$}. 
\end{aligned}
\end{equation}
\end{enumerate}
Moreover, the above radii $r_1$, $r_2$ and $r_3$ can be chosen so that 
\[ \lim_{r_1\uparrow 1}\frac{\log(r_2)}{\log(r_1)}=\tau'_{\mathrm{min}}
\text{ and }  \lim_{r_1\uparrow 1}\frac{\log(r_3)}{\log(r_1)}=\tau'_{\mathrm{max}}, \]
where $\tau'_{\mathrm{min}} = \min_{z\in \T}|\tau'(z)|$ and $\tau'_{\mathrm{max}} = \max_{z\in \T}|\tau'(z)|$. 
\end{lemma}
 \begin{proof}
 Note that the unit circle $\mathbb{T}$ is an invariant set for the map $\tau$. Then setting 
$\tau(\exp(i t))=\exp(i f(t))$ for $t \in \mathbb{R}$ defines an analytic map $f:\mathbb{R}\to \mathbb{R}$ with
\begin{equation*}
f'(t)=\frac{\tau'(e^{it})}{\tau(e^{it})} e^{it} \, ,
\end{equation*}
see, for example, \cite[Lemma 5.1]{BN}.
We shall only consider the case of orientation-preserving $\tau$ here, which implies  $f'(t)>0$ for $t\in \mathbb{R}$. The orientation-reversing case is similar.  

For later use, we note that setting 
$f'_{\mathrm{min}}=\min_{t\in\mathbb{R}}f'(t)$ and 
$f'_{\mathrm{max}}=\max_{t\in\mathbb{R}}f'(t)$ we have 
$f'_{\mathrm{min}}=\tau'_{\mathrm{min}}$ and $f'_{\mathrm{max}}=\tau'_{\mathrm{max}}$. 

For $z = \exp(i t)$ and $w = r \exp(i t)$ with $\tau$ analytic and non-zero on the line-segment $[z,w]$, we have
\begin{equation*}
\log \tau(z)-\log \tau(w) = \int_{w}^z \frac{\tau'(\zeta)}{\tau(\zeta)} d \zeta
\end{equation*}
so that
\begin{equation*}
\log|\tau(e^{it})| - \log|\tau(r e^{it})| = \operatorname{Re}
\int_{r}^1 \frac{\tau'(\rho e^{it})}{\tau(\rho e^{it})} e^{it} d\rho \, .
\end{equation*}
Since $|\tau(\exp(it))|=1$ we have
\begin{equation}
\label{eq:tauest}
\begin{aligned}
-\log|\tau(r e^{it})| &= \int_{r}^1 f'(t) d\rho+
\operatorname{Re} \int_{r}^1 \left(\frac{\tau'(\rho e^{it})}{\tau(\rho e^{it})} e^{it} - f'(t) \right) d\rho \\
&= (1-r) f'(t) + R(r,t)  \, .
\end{aligned}
\end{equation}
Since the integrand of the second integral is analytic and vanishes at the upper limit, there is a constant $C>0$, such that 
for all  real $r$ sufficiently close to $1$ we have 
\begin{equation}
\label{eq:Rrtest}
\sup_{t\in \mathbb{R}}|R(r,t)|\leq C(1-r)^2\, . 
\end{equation}
Now fix a function $\epsilon:(0,\infty)\to (0,\infty)$ with 
\begin{equation}
\label{eq:eps1}
 0<\epsilon(r)<f'_{\mathrm{min}}-1 \quad (\forall r\in (0,1))
 \end{equation}
such that 
 \begin{equation}
 \label{eq:eps2} 
 \lim_{r\to 1}\epsilon(r)=0 \text{ and } \lim_{r\to 1} \frac{\epsilon(r)}{|1-r|}=\infty\,.
 \end{equation}
 To be definite, we could choose 
 \[ \epsilon(r)=\frac{1}{2}(f'_{\mathrm{min}}-1)\sqrt{|1-r|}\,.\]
Using (\ref{eq:Rrtest}), the fact that $\lim_{r\to 1}\log(r)/(r-1)=1$ and the properties (\ref{eq:eps1}) and (\ref{eq:eps2}) 
of $\epsilon$ we see that there is $\tilde{r}\in (0,1)$ with $\tau$ analytic on $A_{\tilde{r}}$ such that 
\begin{equation}
\label{eq:Rrtest2}
\left | \frac{R(r,t)}{1-r} \right |< \epsilon(r) \quad (\forall r\in (\tilde{r},1)\cup (1,\tilde{r}^{-1}))
\end{equation} 
and 
\begin{equation} 
\label{eq:minest}
f'_{\mathrm{min}} - \epsilon(r) > \frac{-\log(r)}{1-r}\quad (\forall r\in (\tilde{r},1))\,.
\end{equation} 
For $r\in(\tilde{r},1)$ we now define radii 
\[ r_2(r) = \exp(-(1-r)(f'_{\mathrm{min}} -\epsilon(r)))\,, \]
\[ r_3(r) = \exp(-(r^{-1}-1)(f'_{\mathrm{max}} +\epsilon(r)))\,. \]
We shall now show that these radii have the desired properties, that is, for every $r\in (\tilde{r}, 1)$ and every $t\in \mathbb{R}$ we have 
\begin{equation}
\label{eq:first}
r_3(r)<r_2(r)<r\,,
\end{equation}
\begin{equation}
\label{eq:second}
r_3(r)<|\tau(re^{it})| < r_2(r)\,,
\end{equation}
\begin{equation} 
\label{eq:third}
(r_2(r))^{-1}<|\tau(r^{-1}e^{it})| < (r_3(r))^{-1}\,.
\end{equation}
We start by observing that (\ref{eq:first}) follows from (\ref{eq:minest}) and the fact that $1-r< r^{-1}-1$ for $r\in (0,1)$. 
Next note that (\ref{eq:Rrtest2}) implies that for $r\in (\tilde{r},1)$ we have 
\[ (f'_{\mathrm{min}}-\epsilon(r) )<f'(t)+\frac{R(r,t)}{1-r} < \frac{r^{-1}-1}{1-r} (f'_{\mathrm{max}}+\epsilon(r) ) \]
and 
\[ \frac{1-r}{r^{-1}-1}(f'_{\mathrm{min}}-\epsilon(r) )<f'(t)-\frac{R(r^{-1},t)}{r^{-1}-1} < (f'_{\mathrm{max}}+\epsilon(r) ) \]
which together with (\ref{eq:tauest}) imply (\ref{eq:second}) and (\ref{eq:third}).  

Finally, using the definition of the radii $r_2$ and $r_3$ it is not difficult to see that 
 \[ \lim_{r\uparrow1} \frac{\log(r_2(r))}{\log(r)} = f'_{\mathrm{min}} 
 \text{ and }
 \lim_{r\uparrow1} \frac{\log(r_3(r))}{\log(r)} = f'_{\mathrm{max}}\,,\] 
 which finishes the proof. 
 \end{proof}

It turns out that our convergence results hold for a larger class of maps defined below, which do not need to preserve the unit circle. 

\begin{defn} Given radii $0<r_3<r_2<r_1<1$ we say that $\tau:A_{r_1}\to \mathbb{C}$ is \emph{holomorphically $(r_2,r_3)$-expansive on $A_{r_1}$} (or \emph{holomorphically expansive on $A_{r_1}$}, for short) if $\tau$ is holomorphic on the closure of $A_{r_1}$ and either 
\begin{equation}\label{eq:tauop}
\begin{aligned}
r_3^{\hphantom{-1}}<&\abs{\tau(z)}<r_2^{\hphantom{-1}} \quad \text{for any $z\in\C$ with  $|z|=r_1$}, \\
r_2^{-1}<&\abs{\tau(z)}<r_3^{-1}   \quad \text{for any $z\in\C$ with  $|z|=r_1^{-1}$};
\end{aligned}
\end{equation}
or 
\begin{equation}\label{eq:tauor}
\begin{aligned}
r_2^{-1}<&\abs{\tau(z)}<r_3^{-1} \quad \text{for any $z\in\C$ with  $|z|=r_1$}, \\
r_3^{\hphantom{-1}}<&\abs{\tau(z)}<r_2^{\hphantom{-1}}   \quad \text{for any $z\in\C$ with  $|z|=r_1^{-1}$}. 
\end{aligned}
\end{equation}
hold. 
\end{defn}

For an illustration of the role played by the radii $r_1$, $r_2$ and $r_3$ see Figure~\ref{fig:fig_rs}. 

\begin{figure}
  \includegraphics[width=0.5\textwidth]{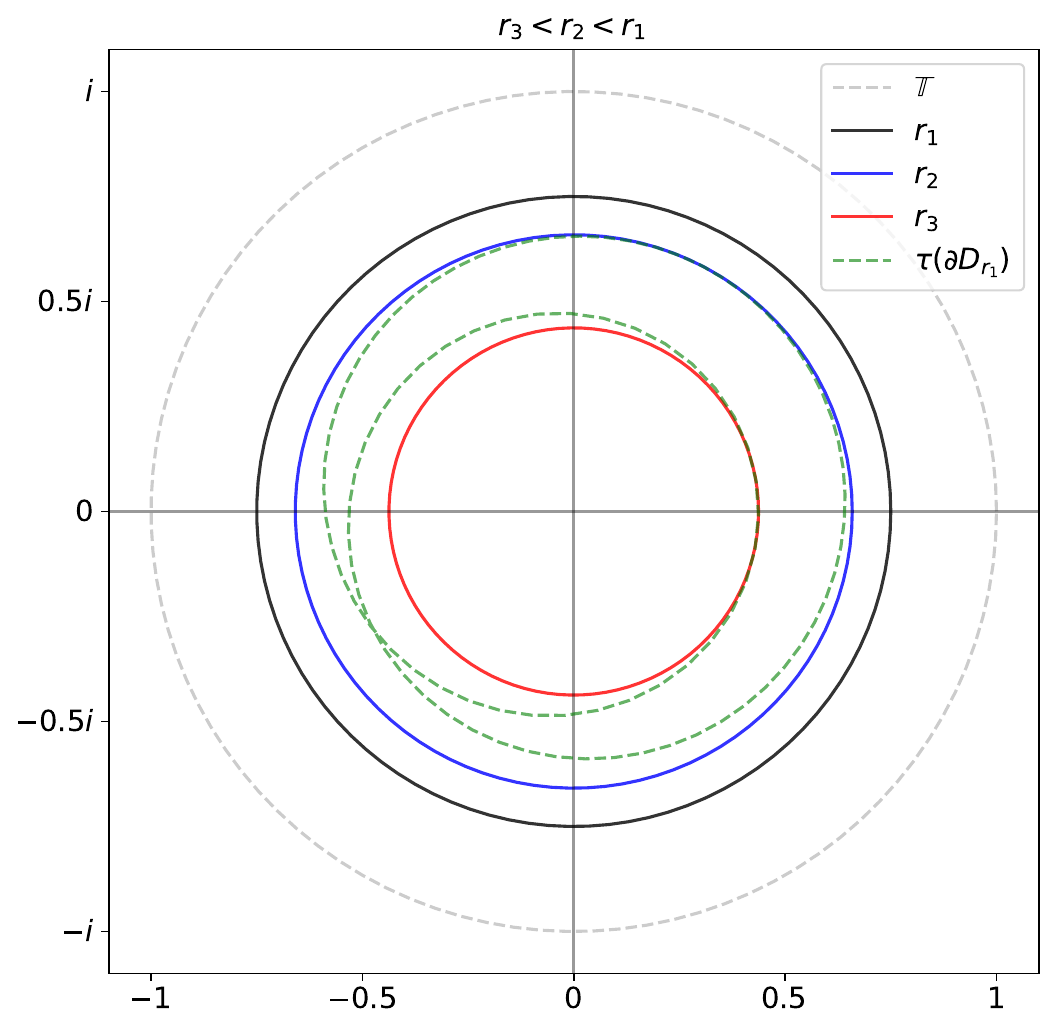}
  \caption{Illustration of radii $r_3 < r_2 < r_1$ for a holomorphically expansive $\tau$ satisfying (\ref{eq:tauop}).}
  \label{fig:fig_rs}
\end{figure}

\begin{remark} \mbox{ } 
\begin{enumerate}
\item 
Note that in order for a $\tau$ which is holomorphic on the closure of an annulus $A_{r_1}$ to be holomorphically expansive on $A_{r_1}$ it suffices that 

either 
\[ \sup_{t\in \mathbb{R}}|\tau(r_1e^{it})|<r_1 \text{ and } r_1^{-1}< \inf_{t\in \mathbb{R}}|\tau(r_1^{-1}e^{it})| \]

or 
\[ \sup_{t\in \mathbb{R}}|\tau(r_1^{-1}e^{it})|<r_1 \text{ and } r_1^{-1}< \inf_{t\in \mathbb{R}}|\tau(r_1e^{it})|\,. \]
\item By Lemma~\ref{lem:compexp} every analytic expanding circle map is holomorphically expansive on any $A_r$ with $r$ sufficiently close to $1$.  
\end{enumerate}
\end{remark}

Our next task will be to define a class of Hilbert spaces on which the Koopman operator of a holomorphically expansive map has good spectral properties. We start by recalling the definition of the Hardy-Hilbert space $H^2(A_r)$ with $r\in (0,1)$ which consists of those functions $f$ holomorphic on the annulus $A_r$ for which 
\[ \sup_{\rho\in (r,1]}\left ( \frac{1}{2\pi}\int_0^{2\pi} |f( \rho e^{it} )|^2\,dt + \frac{1}{2\pi} \int_0^{2\pi} |f( \rho^{-1} e^{it} )|^2\,dt \right ) < \infty\,. \]
It turns out that $H^2(A_r)$ is a Hilbert space with scalar product 
\begin{equation*}
 (f,g)_{H^2(A_r)} =
\frac{1}{2\pi} \int_0^{2\pi} f^*(re^{it})\,\overline{g^*(re^{it})}\,dt +
\frac{1}{2\pi} \int_0^{2\pi} f^*(r^{-1}e^{it})\,\overline{g^*(r^{-1}e^{it})}\,dt\,,
\end{equation*}
where $f^*$ and $g^*$ denote the respective non-tangential boundary values of $f$ and
$g$. More details on this construction can be found in \cite{BaJuSl_AIHP17} or \cite{Sarason1965}. 
For later use we note that $(e_n)_{n\in \mathbb{Z}}$ with
\begin{equation}
\label{eq:eONB}
e_n(z)=\frac{z^n}{\sqrt{r^{2n}+r^{-2n}}} \quad (n\in \mathbb{Z})
\end{equation}
 is an orthonormal basis for $H^2(A_r)$. We shall use this space later on to control the quadrature error inherent in EDMD. 
 
 The space on which we shall study the Koopman operator is denoted by $H^2(A_r^c)$ and is constructed as follows. 
 Let $L^2(\mathbb{T})$ denote the Hilbert space of square-integrable functions on the unit circle. Given $f\in L^2(\mathbb{T})$ we write 
 \[ c_n(f) = \frac{1}{2\pi}\int_0^{2\pi}f(e^{it})e^{-int}\,dt \quad (n\in \mathbb{Z}) \]
 for the Fourier coefficients of $f$. 
 
 Let $r\in (0,1)$ and let $L$ denote the set of all finite linear combinations of positive and negative powers of $z$, that is, $L$ is the set of Laurent polynomials in $z$.  We construct a norm $\|\cdot\|_{H^2(A_r^c)}$ on $L$ as follows 
 \[  \|f\|_{H^2(A_r^c)} = \sqrt{\sum_{n\in \mathbb{Z}} |c_n(f)|^2r^{2|n|}}\,, \]
 and define $H^2(A_r^c)$ to be the completion of $L$ with respect to this norm. It turns out that 
 $H^2(A_r^c)$ is a Hilbert space with scalar product 
 \[ (f,g)_{H^2(A_r^c)} = \sum_{n\in \mathbb{Z}} c_n(f) \overline{c_n(g)}r^{2|n|}\,,\]
and orthonormal basis 
$(e^c_n)_{n\in \mathbb{Z}}$ where
\begin{equation}
\label{eq:ecONB} 
e^c_n(z)=r^{-\abs{n}}z^n \quad (n\in \mathbb{Z})\,.
\end{equation}
The space $H^2(A_r^c)$ is quite large. It contains $L^2(\mathbb{T})$, but also all distributions, and other highly singular objects, such as hyperfunctions. The reason for the notation, which looks peculiar at first sight, stems from the fact that $H^2(A_r^c)$ can be identified with the Hardy-Hilbert space of functions holomorphic on the complement $A_r^c$ of $A_r$ in the Riemann sphere 
$\hat{\mathbb{C}}$, see, for example, \cite{BaJuSl_AIHP17}. 

We shall now show that the  Koopman or composition operator $C_\tau\colon f\mapsto f\circ \tau$ is a well-defined compact operator on $H^2(A_r^c)$. We start with the following crucial lemma. 

\begin{lemma}
\label{lem:me}
Let $\tau$ be holomorphically $(r_2,r_3)$-expansive on $A_{r_1}$ and let $r\in(r_2,r_1)$.Then the following holds: 
\[ C_\tau e^c_k \in H^2(A_r^c) \quad (k\in \mathbb{Z})\,, \]
\[  \abs{(C_\tau e^c_k,e^c_l)_{H^2(A_r^c)}} \leq
\left ( \frac{r_2}{r\vphantom{r_1}} \right )^{\abs{k}}
\left ( \frac{r}{r_1} \right )^{\abs{l}} \quad (k,l\in \mathbb{Z})\,,\]
where $(e^c_n)_{n\in \mathbb{Z}}$ is the orthonormal basis for $H^2(A_r^c)$ given in (\ref{eq:ecONB}). 
\end{lemma} 
\begin{proof}
We start with a simple observation. 
For $\rho>0$ let $\partial D_\rho$ denote the positively oriented boundary of the disk $D_\rho=\all{z\in \mathbb{C}}{|z|<\rho}$. 
For any $\rho \in [r_1,r_1^{-1}]$ we have 
\begin{equation}
\label{eq:contint}
c_l(C_\tau e_k^c) = 
\frac{r^{-\abs{k}}}{2\pi i} \int_{\partial  D_{\rho}}
\tau(z)^k z^{-(l+1)}\,dz \quad (k,l\in \mathbb{Z})\,. 
\end{equation}
This follows by rewriting the definition of the $l$-th Fourier coefficient $c_l$ as a contour integral and then using the fact that $\tau$ is holomorphic on the closure of $A_{r_1}$ to shift the contour. We claim that the above equation implies 
\begin{equation}
\label{eq:me}
r^{\abs{l}}c_l(C_\tau e_k^c)\leq
\left ( \frac{r_2}{r\vphantom{r_1}} \right )^{\abs{k}}
\left ( \frac{r}{r_1} \right )^{\abs{l}} \quad (k,l\in \mathbb{Z})\,.
\end{equation}
The proof of the inequality above splits into two cases depending on whether $\tau$ satisfies (\ref{eq:tauop}) or (\ref{eq:tauor}). We shall only consider $\tau$ satisfying (\ref{eq:tauor}) here, the other case can be treated similarly. 

To start with suppose that $k\geq 0$. Since (\ref{eq:tauor}) holds, we know that for all $z$ with
$\abs{z}=r_1^{-1}$  we have $\abs{\tau(z)}\leq r_2$, and so, since $k\geq 0$ we have
$ \abs{\tau(z)}^k\leq r_2^{k}$ for any $z$ with $\abs{z}=r_1^{-1}$. Using (\ref{eq:contint}) we have
\[  r^{\abs{l}}c_l(C_\tau e_k^c)\leq
\frac{r^{\abs{l}-k}}{2\pi} \int_{\partial D_{r_1^{-1}}}
\abs{\tau(z)}^k \abs{z}^{-(l+1)}\,\abs{dz}
\leq r^{\abs{l}-k}r_2^{k}r_1^{l}\leq
\left ( \frac{r_2}{r\vphantom{r_1}} \right )^{k}
\left ( \frac{r}{r_1} \right )^{\abs{l}} \,,
\]
since $r_1^l\leq r_1^{-\abs{l}}$ for all $l\in \mathbb{Z}$.

Suppose now that $k<0$. By (\ref{eq:tauor}) we know that for all $z$ with
$\abs{z}=r_1$  we have $\abs{\tau(z)}\geq r_2^{-1}$, and so, since $k< 0$ we have
$ \abs{\tau(z)}^k\leq r_2^{-k}$ for any $z$ with $\abs{z}=r_1$. Using (\ref{eq:contint}) we have
\[  r^{\abs{l}}c_l(C_\tau e_k^c)\leq
\frac{r^{\abs{l}+k}}{2\pi} \int_{\partial D_{r_1}}
\abs{\tau(z)}^k \abs{z}^{-(l+1)}\,\abs{dz}
\leq r^{\abs{l}+k}r_2^{-k}r_1^{-l}\leq
\left ( \frac{r_2}{r\vphantom{r_1}} \right )^{-k}
\left ( \frac{r}{r_1} \right )^{\abs{l}} \,,
\]
since $r_1^{-l}\leq r_1^{-\abs{l}}$ for all $l\in \mathbb{Z}$. This finishes the proof of (\ref{eq:me}). The first assertion of the lemma now follows since, using (\ref{eq:me}) we have  
\[ \|C_\tau e_k^c\|_{H^2(A_r^c)}^2=\sum_{l\in \mathbb{Z}}  r^{2\abs{l}}|c_l(C_\tau e_k^c)|^2<\infty\,.\]
The remaining assertion also follows from (\ref{eq:me}), since  $(C_\tau e^c_k,e^c_l)_{H^2(A_r^c)}=r^{\abs{l}}c_l(C_\tau e_k^c)$.   
\end{proof}
An immediate consequence of the previous lemma is that the Koopman operator
$C_\tau$ is Hilbert-Schmidt on $H^2(A_r^c)$. 
\begin{cor}
\label{cor:HS}
Let $\tau$ be holomorphically $(r_2,r_3)$-expansive on $A_{r_1}$ and let $r\in(r_2,r_1)$. Then 
$C_\tau$ is a Hilbert-Schmidt operator from $H^2(A_r^c)$ into itself  
with Hilbert-Schmidt norm bounded by
\[ \|C_\tau\|_{S^2(H^2(A_r^c))}=\sqrt{\sum_{k,l\in \mathbb{Z}}\abs{(C_\tau e^c_k, e^c_l)}
^2}\leq \sqrt{\frac{(r^2+r_2^2)(r_1^2+r^2)}{(r^2-r_2^2)(r_1^2-r^2)}}\,.
\]
\end{cor}
Since $C_\tau$ is compact, its spectrum is a finite set or a sequence converging to zero together with zero itself, and every non-zero spectral point is an eigenvalue of finite algebraic multiplicity.
Next we shall show that spectral data of 
$C_\tau$ on $H^2(A_r^c)$ is effectively approximated by spectral data of matrices $M_{\tau;m,n}$ constructed as follows.
Let a family of complex-valued functions be defined by
\[ f_{\tau;k,l}(z)=\tau(z)^kz^{-l} \quad (k,l\in \mathbb{Z}) \]
and, for $m\in \mathbb{N}$, let $L_m$ denote the following continuous functional on $H^2(A_r)$
\[ L_m(f)=\frac{1}{m} \sum_{l=0}^{m-1}f(
\exp( \frac{2\pi i l}{m}+i\alpha )) \quad (f\in H^2(A_r))\,, \]
where $m\in \N$ and $\alpha \in \R$. As we shall see later $L_m$ provides a
quadrature rule for functions in $H^2(A_r)$ converging to the functional
\[ L(f)=\frac{1}{2\pi}\int_0^{2\pi} f(\exp(it))\,dt \quad (f \in H^2(A_r))\,\]
at exponential speed in the strong dual topology of $H^2(A_r)$.
For $m,n\in \mathbb{N}$ define the $(2n-1)\times (2n-1)$ matrix $M_{\tau;m,n}$  by
\begin{equation}
\label{eq:Mdef}
M_{\tau;m,n}= (L_m(f_{\tau;k,l})_{k,l}) \quad (k,l \in \set{-n+1,\ldots,n-1})\,.
\end{equation}
As we shall see the above two-parameter family of matrices yield effective spectral
approximants for $C_\tau$,
where the parameter $n$ describes the order of truncation of a
Galerkin-type approximation, while the parameter $m$ describes the order of
truncation in a collocation-type method. Numerical experiments suggest
that in order to achieve convergence $m$ has to be chosen at least
twice as large as $n$ in order to achieve convergence.
In the following we shall state and prove a more precise version of this observation.

First recall that by Corollary~\ref{cor:HS} we know that $C_\tau$ is Hilbert-Schmidt on $H^2(A_r^c)$. 
It turns out that $C_\tau$ can be approximated at exponential speed
in Hilbert-Schmidt norm by the
finite-rank Galerkin approximants $P_nC_\tau P_n$, where for $n\in \mathbb{N}$ the
operator $P_n\colon H^2(A^c_r)\to H^2(A^c_r)$ is the orthogonal
projection given by
\[ P_nf=\sum_{\abs{k}<n}(f,e^c_k)_{H^2(A^c_r)} e^c_k\,, \]
as we shall see now. 
\begin{lemma}
\label{lem3}
Let $\tau$ be holomorphically $(r_2,r_3)$-expansive on $A_{r_1}$ and 
let $r\in (r_2,r_1)$. Then there is a constant $K_1$ depending on $r$, $r_1$ and $r_2$ only, such that
\[  \|C_\tau-P_nC_\tau P_n\|_{S^2(H^2(A_r^c))}\leq K_1 \rho^n \quad (n\in \mathbb{N})\,,\]
where $\rho=\max\set{\frac{r_2}{r}, \frac{r}{r_1}}$.
\end{lemma}
\begin{proof}
Writing $c_{k,l}=(C_\tau e^c_k, e^c_l)_{H^2(A_r^c)}$ for $k,l\in \mathbb{Z}$ we have
\[ \|C_\tau-P_nC_\tau P_n\|^2_{S^2(H^2(A_r^c))}=
\sum_{\substack{\abs{k}\geq n\\ \abs{l}<n}} \abs{c_{k,l}}^2
+\sum_{\substack{\abs{k}<n\\ \abs{l}\geq n}} \abs{c_{k,l}}^2
+\sum_{\substack{\abs{k}\geq n\\ \abs{l}\geq n}} \abs{c_{k,l}}^2\,.
 \]
Using Lemma~\ref{lem:me} the first sum can be bounded by
\[ \sum_{\substack{\abs{k}\geq n\\ \abs{l}<n}} \abs{c_{k,l}}^2
\leq \sum_{\substack{\abs{k}\geq n\\ \abs{l}<n}}
\left ( \frac{r_2}{r\vphantom{r_1}} \right )^{2\abs{k}}
\left ( \frac{r}{r_1} \right )^{2\abs{l}} \leq
2\frac{r^2(r_1^2+r^2)}{(r^2-r_2^2)(r_1^2-r^2)} \left( \frac{r_2}{r}\right )^{2n}\,,
\]
while the second sum can be bounded by
\[ \sum_{\substack{\abs{k}< n\\ \abs{l}\geq n}} \abs{c_{k,l}}^2
\leq \sum_{\substack{\abs{k}< n\\ \abs{l}\geq n}}
\left ( \frac{r_2}{r\vphantom{r_1}} \right )^{2\abs{k}}
\left ( \frac{r}{r_1} \right )^{2\abs{l}} \leq
2\frac{r_1^2(r^2+r_2^2)}{(r_1^2-r^2)(r^2-r_2^2)} \left( \frac{r}{r_1}\right )^{2n}\,,
\]
and the third sum by
\[ \sum_{\substack{\abs{k}\geq  n\\ \abs{l}\geq n}} \abs{c_{k,l}}^2
\leq \sum_{\substack{\abs{k} \geq  n\\ \abs{l}\geq n}}
\left ( \frac{r_2}{r\vphantom{r_1}} \right )^{2\abs{k}}
\left ( \frac{r}{r_1} \right )^{2\abs{l}}
\leq 4 \frac{(r^2+r_2^2)(r_1^2+r^2)}{(r^2-r_2^2)(r_1^2-r^2)}
\left ( \frac{r_2}{r\vphantom{r_1}} \right )^{2n}
\left ( \frac{r}{r_1} \right )^{2n}\,,
\]
from which the assertion follows. 
\end{proof}
\begin{remark}
From Corollary~\ref{cor:HS} we know that $C_\tau$ is Hilbert-Schmidt. 
The lemma above implies that $C_\tau$ enjoys an even stronger property, namely, it is of exponential class (see \cite{expoclass}), that is, its sequence of singular values (and hence its sequence of eigenvalues) is bounded from above by a decreasing exponential, in common with other naturally occurring evolution operators associated with holomorphic data considered on spaces of holomorphic functions (see, for example, \cite{BN, decay, ruelle}). 
\end{remark}
Using the definition of $H^2(A_r^c)$ the matrix elements of $C_\tau$ on $H^2(A^c_r)$ can be
written as integrals over the unit circle as follows
\[ (C_\tau e^c_k,e^c_l)_{H^2(A^c_r)}=\frac{r^{\abs{l}-\abs{k}}}{2\pi}\int_0^{2\pi}
\tau(\exp(it))^k\exp(-ilt)\,dt = r^{\abs{l}-\abs{k}} L(f_{\tau;k.l}) \quad (k,l\in \mathbb{Z})\,. \]
For computational purposes the integration over the unit circle given by the functional
$L$ on $H^2(A_r)$ is replaced by
the numerical quadrature functionals $L_m$ on $H^2(A_r)$
introduced earlier \[ L_m(f)=\frac{1}{m} \sum_{l=0}^{m-1}f(
\exp( \frac{2\pi i l}{m}+i\alpha )) \quad (f\in H^2(A_r))\,, \]
where $m\in \N$ and $\alpha \in \R$. We shall now show that $L_m$ converges to $L$ at
exponential speed in the strong dual topology of $H^2(A_r)$.

\begin{lemma}
\label{lem4}
Let $\rho\in (0,1)$. Then for any $m\in \mathbb{N}$ the $m$-th order
quadrature error for functions $f\in H^2(A_\rho)$ is bounded by
\[
|L(f)-L_m(f)|
\leq \frac{\|f\|_{H^2(A_\rho)}}{\sqrt{1-\rho^{2}}}\rho^{m}\,.
\]
\end{lemma}

\begin{proof}
Recall that
\[ e_k(z)=\frac{z^k}{\sqrt{\rho^{2k}+\rho^{-2k}}} \quad (k\in \Z, z\in A_\rho) \]
is an orthonormal basis in $H^2(A_\rho)$, so
\[ f(z) = \sum_{k\in \Z} \frac{(f,e_k)_{H^2(A_\rho)}}{\sqrt{\rho^{2k}+\rho^{-2k}}} z^k \quad
(f\in H^2(A_\rho), z\in A_\rho)\,. \]
Thus for any $f\in H^2(A_\rho)$ we have
\[ L(f) = \sum_{k\in \Z} \frac{(f,e_k)_{H^2(A_\rho)} }{\sqrt{\rho^{2k}+\rho^{-2k}}}
\frac{1}{2\pi}\int_0^{2\pi}\exp(ikt)\,dt = \frac{1}{\sqrt{2}}(f,e_0)_{H^2(A_\rho)}\,. \]
and
\[ L_m(f) =
\sum_{k\in \Z} \frac{ (f,e_k)_{H^2(A_\rho)} }{\sqrt{\rho^{2k}+\rho^{-2k}}} \frac{1}{m}
\sum_{l=0}^{m-1}
\exp( \frac{2\pi i l k}{m}+ik\alpha)
= \sum_{k\in \Z} \frac{ (f,e_{km})_{H^2(A_\rho)} }{\sqrt{\rho^{2km}+\rho^{-2km}}}
\exp(ikm\alpha)
  \]
since
\[   \frac{1}{m}\sum_{l=0}^{m-1}
\exp( \frac{2\pi i l k}{m}+ik\alpha) = \begin{cases}
\exp(ik\alpha) & \text{ if $k\in m\Z$,} \\
0                     & \text{ if $k \not \in m\Z$.}
\end{cases}
\]
Thus for any $m\in \mathbb{N}$ the $m$-th order  quadrature error is given by
\begin{multline*}
|L(f)-L_m(f)|=\left |  \sum_{k\in \Z\setminus \set{0} }
\frac{ (f,e_{km})_{H^2(A_\rho)} }{\sqrt{\rho^{2km}
+\rho^{-2km}}}
\right | \\
\leq  \sqrt{\sum_{k\in \Z \setminus \set{0}} \frac{1 }{\rho^{2km}+\rho^{-2km}}} \,\|f\|
_{H^2(A_\rho)} \leq \frac{\|f\|_{H^2(A_\rho)}}{\sqrt{1-\rho^{2}}}\rho^{m}\,
\end{multline*}
as required. 
\end{proof}
For $m,n\in \mathbb{N}$ let $C_{\tau;m.n}$ denote the finite-rank operator on
$H^2(A^c_r)$  given by
\begin{equation}
\label{eq:CsimM}
 (C_{\tau;m,n}e^c_k, e^c_l)_{H^2(A^c_r)} = r^{\abs{l}-\abs{k}}L_m(f_{\tau;k,l})
\quad (k,l\in \set{-n+1,\ldots, n-1})\,.
\end{equation}
The operator above has the same non-zero eigenvalues counting algebraic multiplicities
as the matrix $M_{\tau;m,n}$ introduced earlier. Moreover, by Lemma~\ref{lem4}, for
fixed $n$, the sequence  $(C_{\tau;m,n})_{m\in \mathbb{N}}$ converges to
$P_nC_\tau P_n$ in Hilbert-Schmidt norm.

\begin{lemma}
\label{lem5}
Let $\tau$ be holomorphically $(r_2,r_3)$-expansive on $A_{r_1}$ and 
let $r\in (r_2,r_1)$. Then there is a constant $K_2$ depending on $r$, $r_1$ and $r_3$
only such that
\[ \|P_nC_\tau P_n - C_{\tau;m,n} \|_{S^2(H^2(A^c_r))} \leq K_2 \frac{r_1^m}{(rr_3)^n} \quad
(m,n\in \mathbb{N})\,.\]
\end{lemma}

\begin{proof}
Using Lemma~\ref{lem4} it follows that for any $m,n\in \mathbb{N}$
\begin{align*}
\|P_nC_\tau P_n - C_{\tau;m,n} \|^2_{S^2(H^2(A^c_r))} & =
\sum_{\substack{\abs{k}<n \\  \abs{l}<n}}
r^{2\abs{l}-2\abs{k}}\abs{L(f_{\tau;k,l})-L_m(f_{\tau;k,l})}^2 \\
& \leq \sum_{\substack{\abs{k}<n \\  \abs{l}<n}}
r^{2\abs{l}-2\abs{k}} \frac{\|f_{\tau;k,l}\|^2_{H^2(A_{r_1})}}{1-r_1^2}r_1^{2m}\,.
\end{align*}
Since $\tau$ is holomorphically $(r_2,r_3)$-expansive on $A_{r_1}$ we have 
\[ \|f_{\tau;k,l}\|_{H^2(A_{r_1})} \leq \sqrt{2}r_3^{-\abs{k}} r_1^{-\abs{l}} \quad (k,l\in \mathbb{Z})\,. \]
In order to see this suppose for the moment that $\tau$ satisfies (\ref{eq:tauop}). 
Then for $k\geq 0$ we have the bound
\[  \|f_{\tau;k,l}\|^2_{H^2(A_{r_1})} \leq r_2^{2k} r_1^{-2l}+r_3^{-2k}r_1^{2l}
\leq 2r_3^{-2k}r_1^{-2\abs{l}}\,,\]
and for $k<0$ the bound
\[  \|f_{\tau;k,l}\|^2_{H^2(A_{r_1})} \leq r_3^{2k} r_1^{-2l}+r_2^{-2k}r_1^{2l}
\leq 2r_3^{2k}r_1^{-2\abs{l}}\,.\]
The proof for $\tau$ satisfying (\ref{eq:tauor}) is similar.
Thus
\[ \|P_nC_\tau P_n - C_{\tau;m,n} \|^2_{S^2(H^2(A^c_r))}
\leq \frac{4(r_1^2+r^2)}{(1-r_1^2)(r_1^2-r^2)(1-(rr_3)^2)}\frac{r_1^{2m}}{(rr_3)^{2n}}\,,
\]
where we have used the bound
\[ \sum_{\abs{k}<n}\alpha^{-\abs{k}}
=1+\frac{2}{\alpha}\frac{\alpha^{-(n-1)}-1}{\alpha^{-1}-1}
\leq \frac{2}{1-\alpha}\alpha^{-n}\,,\]
with $\alpha\in (0,1)$.
\end{proof}
Looking back at what we have achieved so far, we see that for $\tau$ holomorphically $(r_2,r_3)$-expansive on 
$A_{r_1}$ and $r\in (r_2,r_1)$ the Koopman operator
$C_\tau$ acting on $H^2(A^c_r)$ can be approximated in
Hilbert-Schmidt norm by the operators $C_{\tau;m,n}$.  By Lemma~\ref{lem3} and
Lemma~\ref{lem5}, we have the following upper bound for the error in approximation
\begin{align*}
\|C_\tau-C_{\tau;m,n}\|_{S^2(H^2(A^c_r))}& \leq
\|C_\tau-P_nC_\tau P_n\|_{S^2(H^2(A^c_r))}
+ \|P_nC_\tau P_n-C_{\tau;m,n}\|_{S^2(H^2(A^c_r))} \\
& \leq K_1 \max\set{\frac{r_2}{r},\frac{r}{r_1}}^n + K_2 \frac{r_1^m}{(rr_3)^n} \,.
\end{align*}
The first term, depending on $n$ only, arises from the error of the Galerkin
approximation, while the second term, depending on both $m$ and $n$,
is due to the quadrature error of the collocation method. We shall now seek to minimise
the two sources of error, given the available bounds. The Galerkin truncation error
is minimised if we choose $r=\sqrt{r_1r_2}$. For the quadrature error, we need to
choose $m$ so that the quadrature error is asymptotically not larger than the error
arising from the Galerkin truncation. For each $n \in \mathbb{N}$, we want to choose $m = m(n) \in \mathbb{N}$ such that $m \geq \delta n$.
Having fixed $r=\sqrt{r_1r_2}$, we thus seek $\delta > 0$, so that
\[ \limsup_{n\to \infty}\frac{\|P_nC_\tau P_n-C_{\tau;\lceil \delta n\rceil,n}\|_{S^2(H^2(A^c_r))}}
{\|C_\tau-P_nC_\tau P_n\|_{S^2(H^2(A^c_r))}}
< \infty
\]

Using the bounds obtained in Lemma~\ref{lem3} and
Lemma~\ref{lem5} the above yields
\[ \limsup_{n\to \infty} \frac{K_2 r_1^{\lceil \delta n\rceil }}{K_1r_2^{n}r_3^n}< \infty\,, \]
which is equivalent to
\[ \delta \geq \frac{\log(r_2r_3)}{\log(r_1)}\,.\]

Summarising, we have proven the following theorem, which yields an optimised method to approximate $C_\tau$ in Hilbert-Schmidt norm by the finite rank operators $C_{\tau; m,n}$.  

\begin{theorem}\label{thm:main} 
Let $\tau$ be holomorphically $(r_2,r_3)$-expansive on 
$A_{r_1}$ and $r\in (r_2,r_1)$. Then the Koopman operator $C_\tau$ is Hilbert-Schmidt on $H^2(A_r^c)$
and
\begin{equation}
\label{eq:mainest1}
\|C_\tau-C_{\tau;m,n}\|_{S^2(H^2(A^c_r))}
\leq K_1 \max\set{\frac{r_2}{r},\frac{r}{r_1}}^n + K_2 \frac{r_1^m}{(rr_3)^n} \quad (m,n\in \mathbb{N})
\end{equation}
for some constants $K_1,K_2 > 0$.
Moreover, choosing $r=\sqrt{r_1 r_2}$ and $m \geq \delta n $ with $\delta=\frac{\log(r_2r_3)}{\log(r_1)} $ yields
\[\|{C}_\tau - C_{\tau;m,n} \|_{S^2(H^2(A^c_r))} \leq K \left(\sqrt{\frac{r_2}{r_1}}\right)^{n} \quad (n\in \mathbb{N})\]
for some $K>0$.
\end{theorem}
\begin{proof} 
The first estimate (\ref{eq:mainest1}) follows by combining Lemma~\ref{lem3} and Lemma~\ref{lem5}. The remaining assertion is a simple calculation.  
\end{proof}

The theorem above immediately implies our main theorem.  

\begin{theorem}
\label{thm:main2}
Let $\tau$ be holomorphically $(r_2,r_3)$-expansive on 
$A_{r_1}$ and let $r=\sqrt{r_1r_2}$. Then $C_\tau$
is a well-defined Hilbert-Schmidt operator on $H^2(A_r^c)$. In particular, its spectrum
$\operatorname{spec}(C_\tau)$ is either a finite set or a sequence
converging to zero together with zero itself and each non-zero spectral point is an eigenvalue of finite algebraic multiplicity.

Furthermore, for $m,n \in \mathbb{N}$
and $\alpha \in \mathbb{R}$, let $\{\psi_l\}_{-(n-1) \leq l \leq (n-1)}$ with
$\psi_l(z) = z^{l}$ be the set of
observables, $\{z^{(j)}\}_{0\leq j \leq m-1}$ with $z^{(j)}=e^{2\pi i j/ m + i\alpha}$ be
the set of phase space points 
and $M=M_n$ the $(2n-1) \times (2n-1)$ matrix given in \eqref{eq:M}. If for every $n \in \mathbb{N}$, the number of phase space points $m = m(n) \in \mathbb{N}$ is chosen such that
\begin{equation}\label{eq:delta_cond}
m \geq  \frac{\log(r_2r_3)}{\log(r_1)} n, \nonumber
\end{equation}
then the following hold:
  \begin{enumerate}
  \item \label{main:it:a}
  Any convergent sequence $(\lambda_n)_{n\in\mathbb{N}}$ with
  $\lambda_n \in \operatorname{spec}(M_n)$ converges to a spectral point of $C_\tau$.
  \item \label{main:it:b} Conversely, for any $\lambda \in \operatorname{spec}(C_\tau)$ there exists a sequence
  $(\lambda_n)_{n\in\mathbb{N}}$ with $\lambda_n \in \operatorname{spec}(M_n)$ such that
  $\lambda_n \to \lambda$ as $n\to \infty$. More precisely, if $\lambda$ is an eigenvalue with 
  ascent\footnote{An eigenvalue $\lambda$ of an
operator $T$ is said to have \emph{ascent} $p$,
if $p$ is the smallest integer such that the kernel of $(\lambda I-T)^p$ equals
  that of $(\lambda I-T)^{p+1}$. In particular, if $\lambda$ is algebraically simple,
  then $p = 1$.} $p$, we have 
  \[|\lambda - \lambda_n| = O\left ( \left ( \frac{r_1}{r_2}\right )^{n/{(2p)}} \right )\quad \text{as } n\to \infty.\]
  \item \label{main:it:c} Suppose $\lambda \in \operatorname{spec}(C_\tau)$ is non-zero and $(\lambda_n)_{n\in \N}$ denotes the approximating sequence of eigenvalues of $M_n$ given in (\ref{main:it:b}) above. If $\xi_n=(\xi_{n,-n+1}, \ldots, \xi_{n,0},\ldots, \xi_{n, n-1})^T \in \mathbb{C}^{2n-1}$ is a generalised eigenvector of the transpose $M^T_n$ of $M_n$ normalised so that 
  \[ \sum_{|k|<n}|\xi_{n,k}|^2r^{2|k|}=1,\]
 then setting 
  \[ h_n(z)=\sum_{\abs{k}<n}\xi_{n,k}z^k \quad (z\in \mathbb{C}) \]
  yields a sequence of Laurent polynomials $(h_n)_{n\in \mathbb{N}}$ with 
 \[ \| \mathcal{P}h_n-h_n\|_{H^2(A_r^c)} = O\left ( \left ( \frac{r_1}{r_2}\right )^{n/{2}} \right )\quad \text{as } n\to \infty,\]
where $\mathcal{P}$ is the spectral projection associated to the eigenvalue $\lambda$ of $C_\tau$. 
  \end{enumerate}
\end{theorem}
\begin{proof}
We start with a few observations. First we note that the EDMD matrix $M_n$ defined above coincides with the matrix $M_{\tau;m,n}$ given in (\ref{eq:Mdef}), since the matrix $H$  occurring in the definition of $M$ is the identity in our case. 
Next we observe that by (\ref{eq:CsimM})  
\[ \operatorname{spec}(M_{\tau;m,n})=\operatorname{spec}(C_{\tau;m,n})\,.\]
Finally, defining the operator $J_n:\mathbb{C}^{2n-1}\to H^2(A_r^c)$ by 
\[ J_n: (\xi_{-n+1}, \ldots, \xi_{0},\ldots, \xi_{n-1})^T
\mapsto (z\mapsto \sum_{|k|<n}\xi_kz^k) \]
we see, after a short calculation, that  
\[ J_nM_{\tau;m,n}^T=C_{\tau;m,n}J_n, \]
which in turn implies that for any non-zero $\lambda\in \mathbb{C}$ and any $\nu\in \mathbb{N}$  
\[ J_n (\ker((\lambda I-M_{\tau;m,n}^T)^\nu) )= \ker((\lambda I-C_{\tau;m,n})^\nu)\,.\]
Thus, $J_n$ maps generalised eigenvectors of  $M_{\tau;m,n}^T$ corresponding to a non-zero eigenvalue $\lambda$ to generalised eigenvectors of $C_{\tau;m,n}$ corresponding to the eigenvalue $\lambda$. 

The assertions now follow from Theorem~\ref{thm:main} and standard perturbation theory. To be precise, assertions (\ref{main:it:a}) and (\ref{main:it:b})
follow from \cite[Corollaries 2.7, 2.13]{Ahues}; for the bound on the convergence, see 
Theorems 2.17, 2.18, and ensuing remarks in \cite{Ahues}. Finally, \cite[Proposition 2.9]{Ahues} yields
statement (\ref{main:it:c}).
\end{proof}

If the map $\tau$ is additionally assumed to preserve the unit circle, we obtain
the following result, which yields a practical numerical approximation scheme of the Koopman operator without having to compute the quantities $r_1$, $r_2$ and $r_3$.

\begin{theorem}
\label{thm:main3}
Let $\tau$ be an analytic expanding circle map with $\tau'_{\mathrm{max}}$ the maximal derivative of $\tau$ on $\mathbb{T}$.
Furthermore, for $n\in \mathbb{N}$ let $M=M_n$ be the $(2n-1) \times (2n-1)$ EDMD matrix constructed from the $2n-1$ observables and $m\in\mathbb{N}$ phase space points specified in Theorem~\ref{thm:main2}. 

Suppose now that $\delta$ is any positive real number with  
\[ \delta > 1+ \tau'_{\mathrm{max}}\,.\]
Then there is $r\in (0,1)$  such that $C_\tau$ is Hilbert-Schmidt on $H^2(A_r^c)$. 
Moreover, if for every $n \in \mathbb{N}$, the number of phase space points $m = m(n) \in \mathbb{N}$ is chosen such that
\begin{equation}\label{eq:m2n_cond3}
m \geq \delta n,
\end{equation}
then the following holds. 
  \begin{enumerate}
  \item \label{main3:it:a}
  Any convergent sequence $(\lambda_n)_{n\in\mathbb{N}}$ with
  $\lambda_n \in \operatorname{spec}(M_n)$ converges to a spectral point of $C_\tau$.
  \item \label{main3:it:b} Conversely, for any $\lambda \in \operatorname{spec}(C_\tau)$ there exist a sequence
  $(\lambda_n)_{n\in\mathbb{N}}$ with $\lambda_n \in \operatorname{spec}(M_n)$ such that
  $\lambda_n \to \lambda$ as $n\to \infty$. Moreover,
  \[|\lambda - \lambda_n| = O(e^{-an}) \quad \text{as } n\to \infty,\] for some $a>0$. 
    \item \label{main3:it:c} Suppose $\lambda \in \operatorname{spec}(C_\tau)$ is non-zero and $(\lambda_n)_{n\in \N}$ denotes the approximating sequence of eigenvalues of $M_n$ given in (\ref{main3:it:b}) above. If $\xi_n=(\xi_{n,-n+1}, \ldots, \xi_{n,0},\ldots, \xi_{n, n-1})^T \in \mathbb{C}^{2n-1}$ is a generalised eigenvector of the transpose $M^T_n$ of $M_n$ normalised so that 
  \[ \sum_{|k|<n}|\xi_{n,k}|^2r^{2|k|}=1,\]
 then setting 
  \[ h_n(z)=\sum_{\abs{k}<n}\xi_{n,k}z^k \quad (z\in \mathbb{C}) \]
  yields a sequence of Laurent polynomials $(h_n)_{n\in \mathbb{N}}$ with 
 \[ \| \mathcal{P}h_n-h_n\|_{H^2(A_r^c)} = O(e^{-bn}) \quad \text{as } n\to \infty,\]
where $\mathcal{P}$ is the spectral projection associated to the eigenvalue $\lambda$ of $C_\tau$ and some $b>0$. 
  \end{enumerate}
\end{theorem}
\begin{proof}
Fix $\delta>1+\tau'_{\mathrm{max}}$. By Lemma~\ref{lem:compexp} 
there are $r_1$, $r_2$ and $r_3$ with $0<r_3<r_2<r_1<1$ such that 
$\tau$ is holomorphically $(r_2,r_3)$-expansive on $A_{r_1}$ and such that 
\[ \delta> 1 + \frac{\log(r_3)}{\log(r_1)}\,.\]
Next choose $r\in(r_2,r_1)$ such that 
\[ \delta >  \frac{\log(r)}{\log(r_1)} + \frac{\log(r_3)}{\log(r_1)}\,,\]
and finally $\gamma>0$ with 
\[ \delta >  \gamma+ \frac{\log(r)}{\log(r_1)} + \frac{\log(r_3)}{\log(r_1)}\,.\]
Thus, if $m\geq \delta n$ we have $r_1^m\leq r_1^{\gamma n}(rr_3)^n$ and so, by Theorem~\ref{thm:main}, 
\[ 
\|C_\tau-C_{\tau;m,n}\|_{S^2(H^2(A^c_r))}
\leq K_1 \max\set{\frac{r_2}{r},\frac{r}{r_1}}^n + K_2 r_{1}^{\gamma n} \leq K e^{-cn} \quad (n\in \mathbb{N})\,. 
\]
for some positive constants $K_1$, $K_2$ and $K$ and $c>0$. 

The assertions of the theorem now follow as in the proof of Theorem~\ref{thm:main2} from standard perturbation theory. 
\end{proof}

\section{Numerical illustration of convergence properties}\label{sec:3}

The purpose of this section is to demonstrate the practical usefulness of the bound $m \geq \delta n $,
which establishes the connection between the number of phase space points $m$ and the number of observables $n$.  
For any holomorphically expansive map $\tau$ on the annulus, there is an open interval $I_{r_1} = (r^{-}_1, r^{+}_1)$
with $r_1 \in I_{r_1}$ so that $\tau$ satisfies either 
(\ref{eq:tauop}) or (\ref{eq:tauor}), and hence our main Theorem \ref{thm:main2}
holds with $\delta$ being a function of $r_1$. Whilst this theorem holds for all $r_1\in I_{r_1}$ and the corresponding $\delta$,
for practical purposes it is of interest to know the smallest such $\delta$, which we denote by $\delta_{\mathrm{min}}$,
as it determines the minimal number $m$ of equidistant phase space points required for a given set of $n$ observables.

We shall next look at a particular class of nontrivial maps $\tau$, perform a numerical search of $r_1 \in I_{r_1}$ to find a suitable $\delta$, and
empirically observe convergence or non-convergence of the respective eigenvalue approximations obtained using the EDMD algorithm
to the true eigenvalues of $C_\tau$, depending on the chosen $\delta$.
We consider the family of maps
\begin{equation}\label{eq:mab_abc}
  \tau(z) = b_a(z) b_b(z) + c,
\end{equation}
where $b_\mu(z) = \frac{z-\mu}{1-\overline{\mu} z}$ is a M\"obius map for $\mu \in \mathbb{C}$ with $|\mu| < 1$, and $c \in \mathbb{C}$.

\subsection{Convergence for circle maps}
We shall first focus on the case $c=0$. In this case
these maps belong to the class of finite Blaschke products, which can be viewed as maps on $\mathbb{C}\cup \{\infty\}$,
giving rise to expanding maps when restricted to the circle. If additionally $a=b=0$ then $\tau$
reduces to the map $z\mapsto z^2$, which is the usual doubling map when considered on the unit circle. For non-zero
$a,b$ one can think of $\tau$ as a (strong) analytic perturbation of the doubling map. The main reason for considering this class of maps for numerical
experiments is that the entire spectrum of the Koopman operator $C_\tau$ considered on a suitable space
$H^2(A_r^c)$ is known, being determined by the fixed point properties of $\tau$ inside the unit disk \cite{SlBaJu_CNSNS20}. In particular, the
spectrum of $C_\tau$ is given by
\begin{equation}\label{eq:spec_explicit}
\operatorname{spec}(C_\tau) = \{0, 1\} \cup \{\lambda(z_0)^n\colon n\in \mathbb{N}\} \cup \{ \lambda(z_\infty)^n \colon n\in \mathbb{N}\},
\end{equation}
where $\lambda(z_0)$ and $\lambda(z_\infty)$ are the multipliers of the unique attracting fixed points
$z_0 \in D_r$ and $z_\infty \in D_{1/r}^\infty$. Moreover, the non-zero subleading eigenvalues come as complex-conjugate pairs, as in
this case we have $\lambda(z_\infty) = \overline{\lambda(z_0)}$.

Having access to the eigenvalues of $C_\tau$ allows us to compare the distance
between exact and approximated eigenvalues. For this, we measure the distance between the first
subleading eigenvalue $\lambda_2 = \lambda(z_0)$ and the corresponding eigenvalue $\lambda^{nm}_2$ of $C_{\tau;m, n}$ represented
by the $(2n-1)\times (2n-1)$ matrix $M_n$ computed in the EDMD algorithm. 
We choose $a=b=0.33\exp(i \pi/25)$ for the map \eqref{eq:mab_abc}, the same setting as considered in \cite{SlBaJu_CNSNS20}.
The results are presented in Figure \ref{fig:fig_nm}, where for each $m\in [2, 400]$ and $n\in [2, 100]$ we plot the error $|\lambda^{nm}_2 - \lambda_2|$
in log scale, with the darkness of the colour corresponding to the size of the error.

\begin{figure}
  \includegraphics[width=0.75\textwidth]{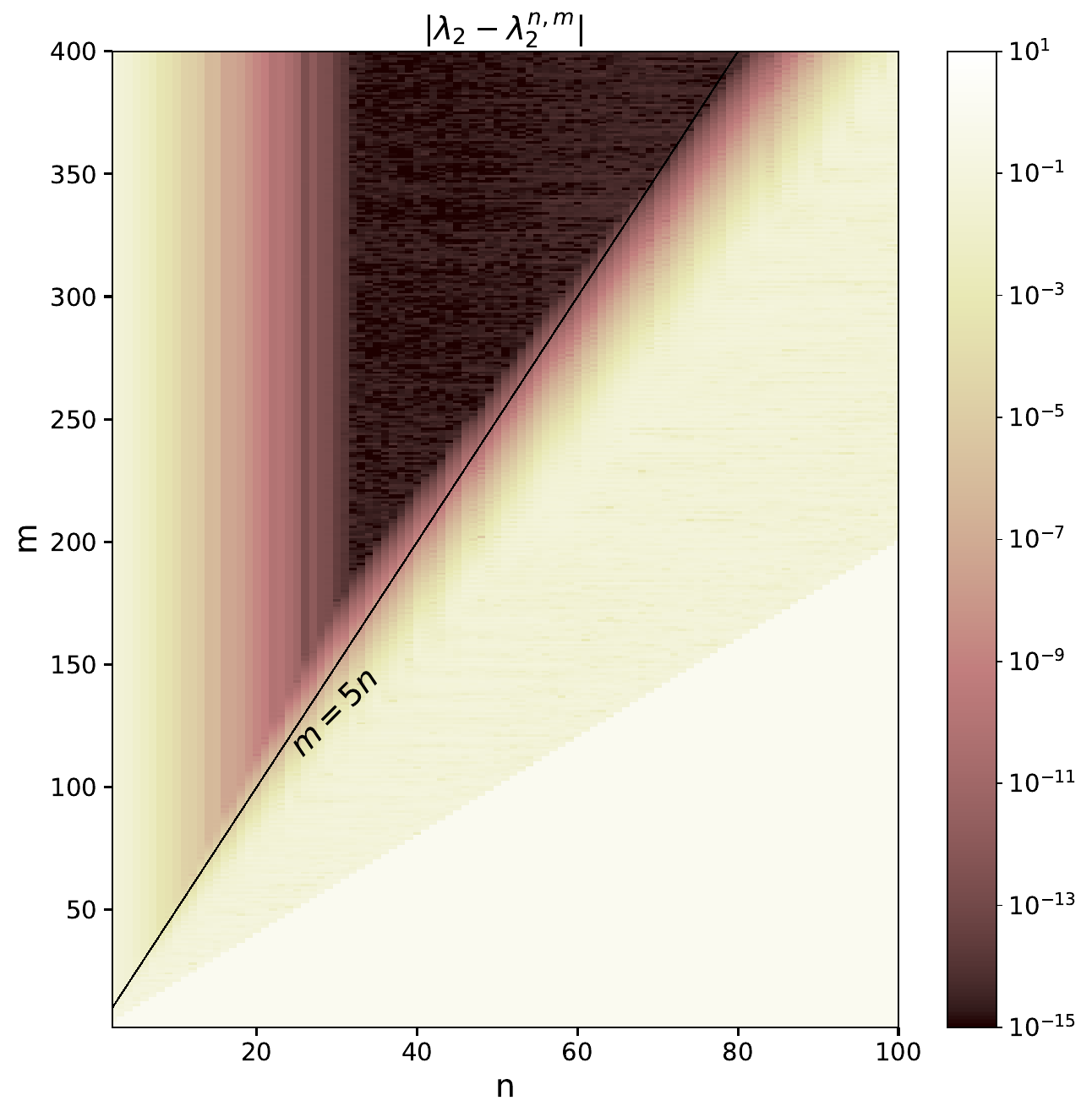}
  \caption{Absolute error of the first subleading eigenvalue (plotted in log scale) for the circle map in (\ref{eq:mab_abc}) with 
  $a=b=0.33\exp(i \pi/25)$ and $c=0$ obtained from the
  EDMD matrix for different values of $m$ and $n$. The black line in the middle is $m = \delta n$
  for $\delta = 5$ nicely indicates our theoretical bound, that is choosing $m \geq 5n$,
  we are seeing exponential convergence in $n$ of the first subleading eigenvalue.}
  \label{fig:fig_nm}
\end{figure}

Using Theorem~\ref{thm:intro}, we have computed an admissible value for $\delta$
as $\tau'_{\mathrm{min}}+\tau'_{\mathrm{max}} \approx 4.98$, which also coincides with $\delta_{\mathrm{min}}$ obtained using a numerical search of $r_1 \in I_{r_1}$ as described above.
The line $m = 5 n$ depicted in Figure \ref{fig:fig_nm} aligns well with the (numerically observed) boundary between convergence and non-convergence.
Here the value $\delta = 5$ was chosen merely for convenience, as this appears to be a sharp bound for integer-valued $\delta$.

In Figure \ref{fig:fig_n_different_m} we plot the error
$|\lambda^{nm}_2 - \lambda_2|$ as a function of $n$ for different regimes of $m = m(n)$.
Choosing a fixed large value of $m$, for example $m = 300$, initially results in exponentially fast convergence in $n$,
up to numerical precision. The error however deteriorates significantly as we increase $n$, here $n > 55$, while holding
$m$ constant, as there are too few phase space points to resolve highly oscillatory modes.
Choosing $m$ as a linear function of $n$, that is $m = \delta n$, we observe for $\delta = 5 \approx \delta_{\mathrm{min}}$ exponential
convergence of the error, albeit slower than the initial convergence for a fixed regime of $m=300$. Choosing $m = 6 n > \delta_{\mathrm{min}} n$,
we observe fast exponential convergence with the same (initial) rate as for a fixed regime $m=300$, which does not deteriorate.
Choosing $\delta$ too small, here $m = 4n < \delta_{\mathrm{min}} n $ we observe no convergence of the error. 
To summarize, the threshold value provided by Theorem~\ref{thm:intro} is not only sufficient for guaranteeing convergence,
but appears to be close to $\delta_{\mathrm{min}}$, that is the boundary between convergence and non-convergence for this example.

\begin{figure}
  \includegraphics[width=0.75\textwidth]{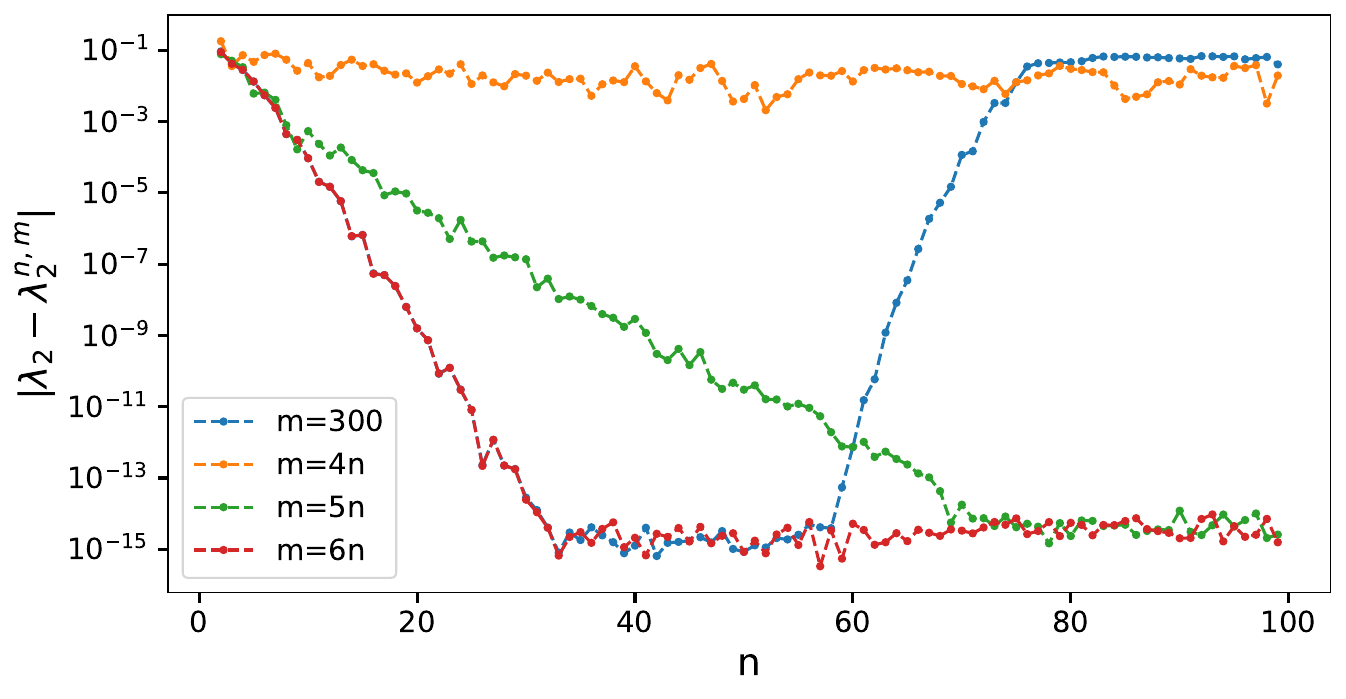}
  \caption{Absolute error of the first subleading eigenvalue (plotted in log scale) for the circle map in (\ref{eq:mab_abc}) with 
  $a=b=0.33\exp(i \pi/25)$ and $c=0$
  obtained from the EDMD matrix, plotted as a function of $n$ for different regimes of choosing $m$.}
  \label{fig:fig_n_different_m}
\end{figure}

\subsection{Convergence for maps not preserving the unit circle}

Theorem~\ref{thm:main2} also applies to holomorphically expansive maps on the annulus which do not necessarily preserve
the unit circle $\mathbb{T}$. To demonstrate our bounds in this setting we consider maps $\tau$ in \eqref{eq:mab_abc} with $c\neq 0$, in which case $\tau$ no longer preserves the unit circle.
We shall choose $a,b$ and $c$ in such a way that $\tau$ is holomorphically expansive on an annulus containing the unit circle,
and whose Julia set is a quasicircle\footnote{
If $a=b=0$ then $\tau(z) = z^2 + c$ is the well-known quadratic family.
If $c \in \{z\in \mathbb{C} \colon |1-\sqrt{1-4z}| < 1\}$ then $\tau$ is hyperbolic and its Julia set is
a quasicircle, that is, the image of a circle under a quasi-conformal map of $\hat{\mathbb{C}}$.}.
Using the same arguments as in the proof of \cite[Theorem 5.4(b)]{BaJuSl_AIHP17},
one can show that the spectrum of $C_\tau$ considered
on a suitable space $H^2(A_r^c)$ is again given by \eqref{eq:spec_explicit}, that is, given by the powers of the multipliers of the two
fixed points in the two respective Fatou components of $\tau$ (however the two multipliers are no longer complex conjugates of each other).

In the same vein as before, in Figure \ref{fig:quasicircle}(right) we show the
approximation error $|\lambda^{nm}_2 - \lambda_2|$ for $n \in [2, 100]$ and different choices of $m = m(n)$. We note that since $\tau$ does not preserve the unit circle, the chosen phase space points are typically in its Fatou set, which tends to make the computations numerically less stable, as the size of the matrix entries grows with $n$. To obtain good approximations, the use of high-precision computation is therefore required in this case. As this involves substantially higher computational cost, we limit ourselves to a sparse grid of $n$ values in the plots provided. Qualitatively, we obtain a very similar picture to the previous case, with a value of $\delta \approx 4.18$ (obtained as before, using a numerical search of $r_1 \in I_{r_1}$) again appearing to be close to $\delta_{\mathrm{min}}$.
We note that the approximation error in Figure \ref{fig:quasicircle} can not improve below $\approx 10^{-16}$,
as our reference value $\lambda_2$ is only computed to this regular floating point precision.

\begin{figure}
  \includegraphics[width=0.3\textwidth]{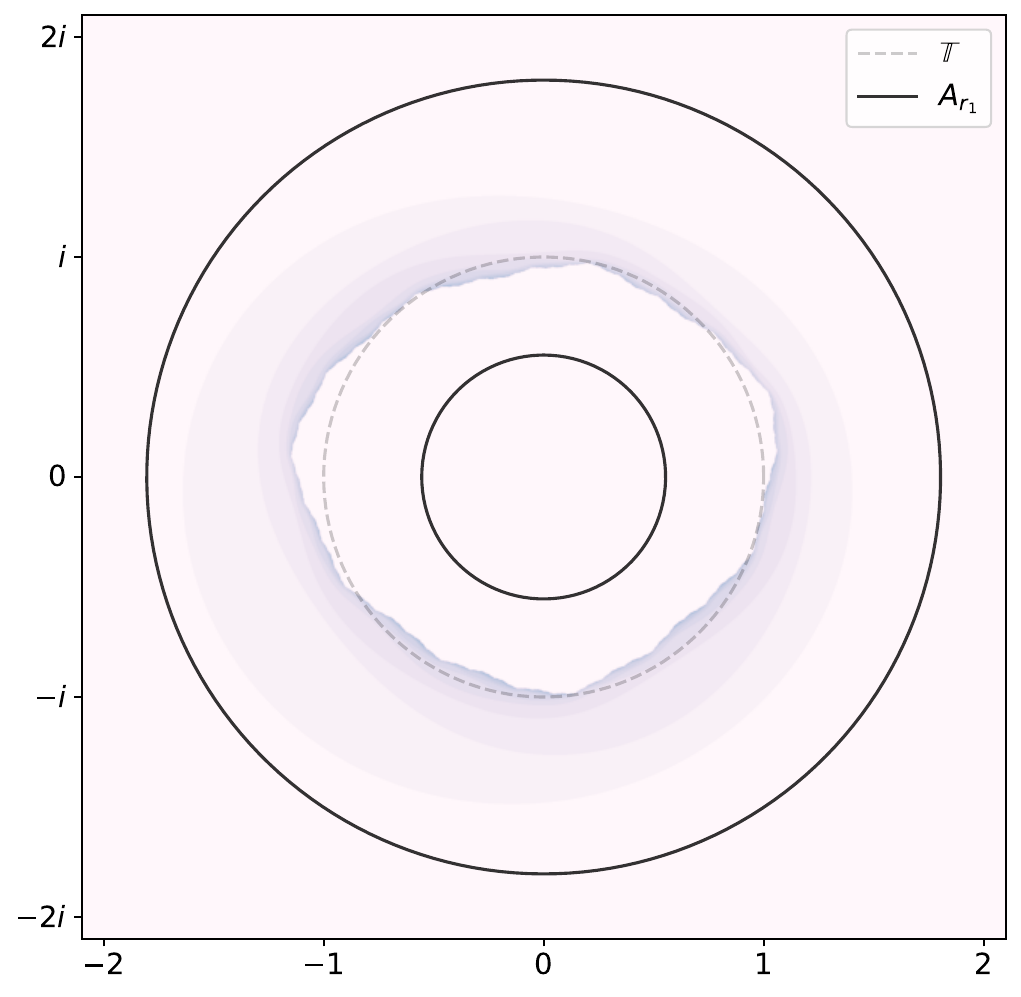}
  \includegraphics[width=0.57\textwidth]{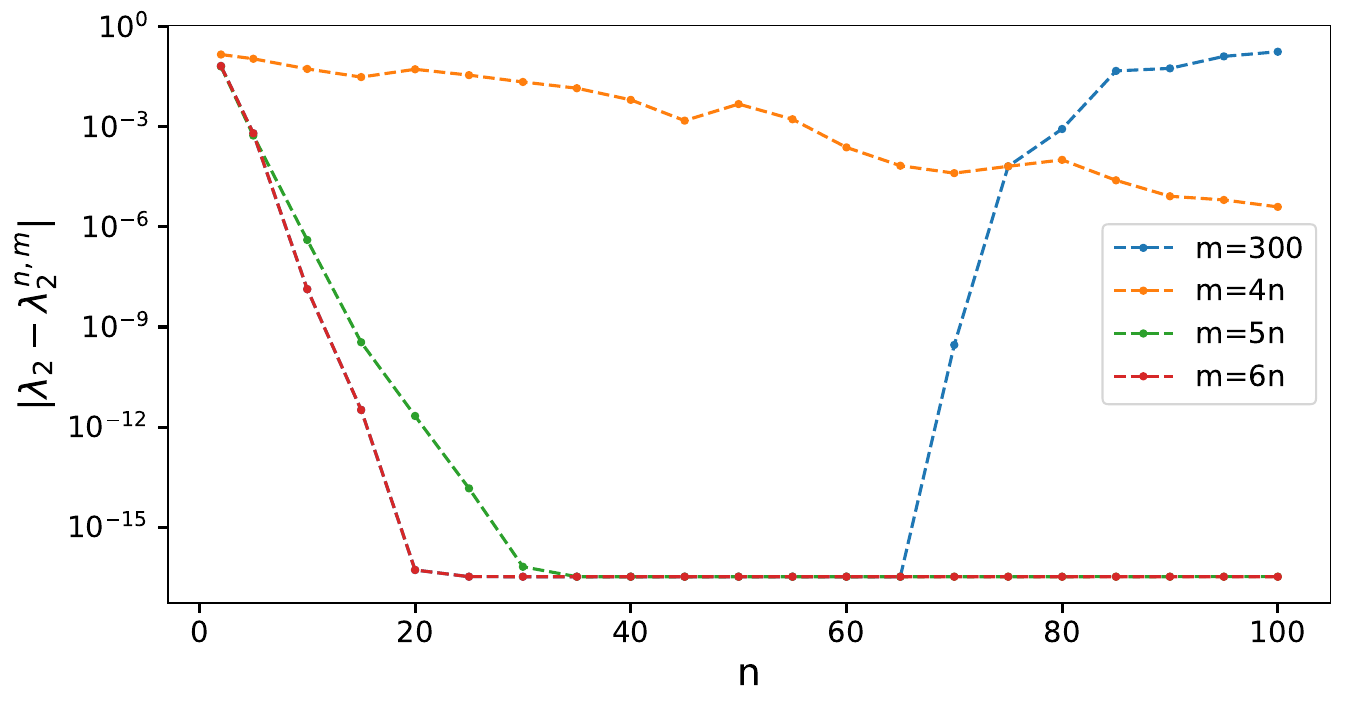}
  \caption{(Left) The Julia set $\mathcal{J}_\tau$ for $\tau$ in \eqref{eq:mab_abc} for
  $a=0.25i+0.11, b=0.078$ and $c=0.164$. The annulus $A_{r_1}$ containing $\mathcal{J}_\tau$
  for $r_1\approx 0.554$ satisfies condition (\ref{eq:tauop}) with corresponding $\delta \approx 4.18$. (Right) EDMD approximation error of the first subleading eigenvalue, as a function of $n$
  for different regimes of $m = m(n)$.}
  \label{fig:quasicircle}
\end{figure}

\section{Acknowledgments}
All authors gratefully acknowledge the support for the research presented in this
article by the EPSRC grant EP/RO12008/1. J.S.~was partly supported by the ERC-Advanced
Grant 833802-Resonances and W.J.~acknowledges funding by the Deutsche Forschungsgemeinschaft (DFG, German Research
Foundation) SFB 1270/2 - 299150580.

\section{Declaration}
The authors have no relevant financial or non-financial interests to disclose.

\bibliographystyle{auth_tit}

\end{document}